\documentclass[11pt]{article}

\usepackage[margin=1in]{geometry}
\usepackage{amsmath,amssymb,amsthm,bm,mathtools}
\usepackage{enumitem}
\usepackage{microtype}
\usepackage{booktabs}
\usepackage{xcolor}
\usepackage[colorlinks=true,linkcolor=blue,citecolor=blue,urlcolor=blue]{hyperref}
\usepackage{comment}
\usepackage{parskip}

\newtheorem{definition}{Definition}[section]
\newtheorem{assumption}[definition]{Assumption}
\newtheorem{lemma}[definition]{Lemma}
\newtheorem{theorem}[definition]{Theorem}
\newtheorem{corollary}[definition]{Corollary}

\theoremstyle{remark}
\newtheorem{remark}[definition]{Remark}

\newcommand{\R}{\mathbb{R}}
\newcommand{\Zp}{\mathbb{Z}_{\ge 0}}
\newcommand{\B}{\mathfrak{B}}
\newcommand{\W}{\mathbb{W}}
\newcommand{\col}{\operatorname{col}}
\newcommand{\im}{\operatorname{im}}
\newcommand{\rank}{\operatorname{rank}}
\newcommand{\id}{\operatorname{id}}
\newcommand{\cD}{\mathcal{D}}
\newcommand{\cU}{\mathcal{U}}

\newcommand{\dd}{\mathrm{d}}

\title{Trajectory Manifolds for Nonlinear Data-Enabled Predictive Control\\
\large Part I: Existence, Smoothness, and Intrinsic Dimension}
\author{Arda Bayer}
\date{\today}

\begin{document}
\maketitle

\begin{abstract}
This note establishes a geometric foundation for trajectory-manifold representations of deterministic nonlinear systems in a behavioral setting motivated by data-enabled predictive control.  For a discrete-time system $x_{k+1}=f(x_k,u_k)$ with measured state and a $C^r$ transition map, $r\geq 1$, we consider the terminal-state-augmented finite-horizon behavior consisting of all admissible state-input trajectories over a prediction
horizon $N$.  We prove that this behavior is a $C^r$ embedded submanifold of the ambient trajectory space with intrinsic dimension $n+Nm$, where $n$ and $m$ are the state and input dimensions.  Moreover, the rollout map from the admissible initial-state and input coordinates $(x_0,\mathbf u)$ is a $C^r$ diffeomorphism onto the behavior manifold, providing explicit global smooth coordinates.  This yields a canonical exact encoder--decoder representation and implies that any exact differentiable latent representation of the full behavior must have latent dimension at least $n+Nm$.  The geometric result does not require controllability, stabilizability, or invertibility of the dynamics.  Corresponding results are given for zero-order-hold sampled continuous-time systems and fixed-step numerical transition maps.  These results provide the deterministic geometric foundation for subsequent data-driven approximation and predictive-control development.
\end{abstract}

\vspace{0.5em} \noindent\textbf{AI Assistance Disclosure.} Generative AI assistance, OpenAI GPT-6 Astra, was used in the development and editing of portions of the proofs and exposition in this document. I have independently reviewed and verified all mathematical arguments, claims, citations, and text in the final manuscript. I accept the final manuscript as my own scholarly work and assume full responsibility for its content, including any remaining errors. 
    
\vspace{0.5em}
\noindent\textbf{Acknowledgment.}
I would like to thank the Aazhang Lab at Rice University for stimulating my interest in the inverted pendulum as a toy problem for fully data-driven nonlinear control, which motivated the trajectory-manifold perspective considered in this work.

\vspace{0.5em}
\noindent\textbf{Funding.}
This research is funded in part through an award from the Rice University Provost's TMC Collaborator Fund.

\vspace{0.5em} \noindent\textbf{Keywords:} behavioral systems; nonlinear control; data-driven control; trajectory manifolds; data-enabled predictive control; model predictive control.

\section{Introduction}

The behavioral approach to dynamical systems describes a system through the collection of signal trajectories that are compatible with its laws \cite{willems1979physical,willems1986timeseries,willems1991paradigms, markovsky2006behavioral}. In this framework, a dynamical system is represented by a time axis, a signal space, and a behavior consisting of its admissible trajectories.

For linear time-invariant systems, Willems' Fundamental Lemma provides a finite-data representation of finite-horizon behavior under controllability and persistent-excitation conditions \cite{willems2005pe}. DeePC uses this result to represent finite-horizon trajectories directly from measured data \cite{coulson2019deepc}. The present note uses the same behavioral viewpoint, but considers the geometry of finite-horizon trajectories of deterministic nonlinear systems.

Smooth trajectory manifolds for nonlinear control systems have previously been studied in a continuous-time functional-analytic setting. Hauser and Meyer \cite{hauser1998trajectory} showed that suitable continuous-time state-input trajectories form a $C^r$ Banach manifold and related its tangent spaces to trajectories of the linearized dynamics. This projection-operator viewpoint was subsequently used for trajectory optimization \cite{hauser2002projection}. Notarstefano and Hauser \cite{notarstefano2008curvature} further studied the second-order geometry of these trajectory manifolds using a weighted inner product and a metric-dependent radius of curvature.

Here we consider the terminal-state-augmented finite-horizon behavior of a discrete-time nonlinear system, as well as sampled continuous-time systems with a finite-dimensional input parameterization. Because a trajectory retains its initial state and all $N$ input values, its natural free coordinates are $(x_0,\mathbf u)$. We show that the resulting behavior is a finite-dimensional $C^r$ embedded manifold of dimension
\[
    n+Nm,
\]
and that these free variables provide global smooth coordinates for the behavior.

The central question of this note is therefore:
\begin{quote}
Under what conditions is the finite-horizon behavior of a deterministic
nonlinear system a smooth finite-dimensional manifold, and what is its
dimension?
\end{quote}
The following sections answer this question and develop its immediate
geometric consequences.

\section{Behavioral notation for a nonlinear state-space system}

We adopt Willems' behavioral-system notation
\cite{willems1986timeseries,willems1991paradigms}.  A dynamical system is
described by the triple
\begin{equation}
    \Sigma=(\mathbb{T},\mathbb{W},\B),
\end{equation}
where $\mathbb{T}$ is the time axis, $\mathbb{W}$ is the signal space, and
\[
    \B\subseteq\mathbb{W}^{\mathbb{T}}
\]
is the \emph{behavior}, i.e., the set of trajectories
$w:\mathbb{T}\rightarrow\mathbb{W}$ that are compatible with the laws of
the system.  We consider discrete time and therefore take
$\mathbb{T}=\Zp$.  This is the same behavioral framework employed in the
Fundamental Lemma \cite{willems2005pe} and subsequently in DeePC
\cite{coulson2019deepc}.

We consider the deterministic nonlinear, and discrete-time, state-space system
\begin{equation}
    x(t+1)=f\bigl(x(t),u(t)\bigr),
    \qquad t\in\Zp,
    \label{eq:nonlinear_system}
\end{equation}
where
\[
    x(t)\in\R^n,
    \qquad
    u(t)\in\R^m.
\]
Because the present theory assumes full-state measurement, we choose the behavioral signal to be
\begin{equation}
    w(t)=\col\bigl(u(t),x(t)\bigr)\in\W:=\R^{m+n}.
\end{equation}
Thus, in the terminology of the behavioral framework, the input and measured
state are taken as the external variables of interest.  This is the
full-state analogue of the input/output partition
$w=\col(u,y)$ used in the LTI behavioral formulation and in DeePC
\cite{willems1986timeseries,coulson2019deepc}.

For a prediction horizon $N\in\mathbb Z_{>0}$, define the stacked input and state sequences
\begin{align}
    \mathbf u &:= \col(u_0,u_1,\ldots,u_{N-1})\in\R^{Nm}, \\
    \mathbf x &:= \col(x_0,x_1,\ldots,x_N)\in\R^{(N+1)n}.
\end{align}
The corresponding finite-horizon signal vector is
\begin{equation}
    \mathbf w:=\col(\mathbf u,\mathbf x)
    \in\R^{Nm+(N+1)n}.
    \label{eq:w_definition}
\end{equation}
The input-first ordering in \eqref{eq:w_definition} is chosen to remain consistent with the behavioral convention $w=\col(u,y)$ used in the LTI literature and in DeePC \cite{willems1986timeseries,coulson2019deepc}. The particular ordering of the coordinates does not affect the results below, since reordering the entries of $w$ does not change the trajectory set or its dimension.

In the behavioral literature, a finite-horizon or truncated behavior denotes the restriction of trajectories in $\B$ to a finite time interval \cite{willems2005pe,coulson2019deepc}.  Our control trajectory additionally retains the terminal state $x_N$, as is customary in finite-horizon optimal control.  We therefore distinguish the following terminal-state-augmented finite-horizon behavior from the ordinary restriction of $\B$.

\begin{definition}[Terminal-state-augmented $N$-step behavior]
\label{def:augmented_behavior}
The terminal-state-augmented $N$-step state-input behavior of
\eqref{eq:nonlinear_system} is
\begin{equation}
\B_N^{+}
:=
\left\{
    \col(\mathbf u,\mathbf x)
    \in\R^{Nm+(N+1)n}
    \;\middle|\;
    x_{k+1}=f(x_k,u_k),\quad k=0,\ldots,N-1
\right\}.
\label{eq:finite_behavior}
\end{equation}
When $f$ is not globally defined, $\B_N^{+}$ contains only those trajectories for which every required evaluation of $f$ is well-defined.
\end{definition}

\section{Existence of the trajectory manifold}

\subsection{Assumptions}

We now establish conditions under which the finite-horizon behavior of \eqref{eq:nonlinear_system} admits a finite-dimensional smooth manifold representation.  The result is stated for a general nonlinear discrete-time system and characterizes both the smoothness and the dimension of its
finite-horizon behavior.

\begin{assumption}[Smooth deterministic dynamics]
\label{ass:smooth_dynamics}
Let
\[
    \mathcal O\subseteq\R^{n+m}
\]
be open, and let
\begin{equation}
    f:\mathcal O\rightarrow\R^n
\end{equation}
be of class $C^r$ for some integer $r\ge 1$.
\end{assumption}

The openness of $\mathcal O$ is the standard setting for differentiable maps.  The assumption $r\ge1$ is the minimum regularity needed below to obtain a differentiable embedded manifold. 

\subsection{The admissible free coordinates}

The natural free coordinates of a deterministic finite-horizon trajectory are the initial state and the future inputs.  Define
\begin{equation}
    q:=\col(x_0,u_0,\ldots,u_{N-1})
    \in\R^{n+Nm}.
    \label{eq:q_coordinates}
\end{equation}
Not every $q$ must generate a valid $N$-step trajectory if $f$ is defined only on $\mathcal O$.  We therefore define the admissible parameter domain carefully.

For $q$ as in \eqref{eq:q_coordinates}, set $x_0(q)=x_0$.  Recursively, whenever the previous states are well-defined, set
\begin{equation}
    x_{k+1}(q)=f\bigl(x_k(q),u_k\bigr).
    \label{eq:rollout_recursion}
\end{equation}
Let $\cD_N\subseteq\R^{n+Nm}$ be the set of all $q$ for which
\[
    \bigl(x_k(q),u_k\bigr)\in\mathcal O,
    \qquad k=0,\ldots,N-1.
\]

\begin{lemma}[Openness and smoothness of the rollout domain]
\label{lem:rollout_smooth}
Under Assumption~\ref{ass:smooth_dynamics}, the admissible rollout domain
$\cD_N$ is an open subset of $\R^{n+Nm}$. Moreover, for every
$k=0,\ldots,N$, the rollout state
\[
    x_k:\cD_N\rightarrow\R^n
\]
is of class $C^r$.
\end{lemma}

\begin{proof}
Let
\[
    q=\col(x_0,u_0,\ldots,u_{N-1})\in\R^{n+Nm}.
\]
We introduce intermediate admissible sets $\mathcal{A}_k$ defined by
\[
    \mathcal{A}_k
    :=
    \left\{
        q\in\R^{n+Nm}:
        (x_j(q),u_j)\in\mathcal O,\;
        j=0,\ldots,k-1
    \right\},
    \qquad k\geq 1,
\]
and set
\[
    \mathcal{A}_0:=\R^{n+Nm}.
\]
Thus, $q\in\mathcal{A}_k$ precisely when the dynamics are well-defined for the first $k$ steps of the rollout. Thus, $\mathcal{A}_0$ corresponds to the stage before any state transition has been evaluated, while $\mathcal{A}_k$, $k\geq1$, contains exactly those free coordinates for which the first $k$ transitions are admissible.

The map
\[
    q\mapsto x_0
\]
simply selects the initial-state coordinates of $q$ and is therefore linear,
hence $C^\infty$.

We now proceed recursively. Suppose that for some
$k\in\{0,\ldots,N-1\}$ the set $\mathcal{A}_k$ is open and the state
$x_k(q)$ is well-defined and $C^r$ on $\mathcal{A}_k$.

Define
\[
    \Psi_k:\mathcal{A}_k\rightarrow\R^{n+m},
    \qquad
    \Psi_k(q):=\col\bigl(x_k(q),u_k\bigr).
\]
The map $x_k(q)$ is $C^r$ by the induction hypothesis, and $u_k$ is obtained
by selecting the corresponding input coordinates of $q$. Hence $\Psi_k$ is
$C^r$ and therefore continuous.

The $(k+1)$st state can be evaluated only when the current state-input pair
lies in the domain $\mathcal O$ of the dynamics. We therefore define
\begin{equation}
    \mathcal{A}_{k+1}
    :=
    \left\{
        q\in\mathcal{A}_k:
        \Psi_k(q)\in\mathcal O
    \right\}
    =
    \Psi_k^{-1}(\mathcal O).
\end{equation}
For example,
\[
    \mathcal{A}_1
    =
    \left\{
        q\in\R^{n+Nm}:
        (x_0,u_0)\in\mathcal O
    \right\},
\]
so the first input $u_0$ is restricted whenever required by
$\mathcal O$; it is not assumed to range freely over all of $\R^m$.

Since $\mathcal O$ is open and $\Psi_k$ is continuous,
$\Psi_k^{-1}(\mathcal O)$ is open in $\mathcal{A}_k$. Because $\mathcal{A}_k$ is open in
$\R^{n+Nm}$, it follows that $\mathcal{A}_{k+1}$ is also open in
$\R^{n+Nm}$.

For every $q\in\mathcal{A}_{k+1}$, the next state is well-defined and satisfies
\[
    x_{k+1}(q)
    =
    f\bigl(x_k(q),u_k\bigr)
    =
    f\bigl(\Psi_k(q)\bigr).
\]
Since $f$ is $C^r$ on $\mathcal O$ and $\Psi_k$ is $C^r$, their composition
$x_{k+1}=f\circ\Psi_k$ is $C^r$ on $\mathcal{A}_{k+1}$.

Repeating this argument for $k=0,\ldots,N-1$ gives
\[
    \mathcal{A}_N
    =
    \left\{
        q\in\R^{n+Nm}:
        (x_k(q),u_k)\in\mathcal O,\;
        k=0,\ldots,N-1
    \right\}.
\]
By the definition of the full-horizon admissible rollout domain,
\[
    \mathcal{A}_N=\cD_N.
\]
Hence $\cD_N$ is open.

Moreover, for each $k$, the map $x_k$ is $C^r$ on $\mathcal{A}_k$. Since
\[
    \cD_N=\mathcal{A}_N\subseteq\mathcal{A}_k,
\]
the restriction of $x_k$ to $\cD_N$ is also $C^r$. Therefore every rollout
map
\[
    x_k:\cD_N\rightarrow\R^n,
    \qquad k=0,\ldots,N,
\]
is of class $C^r$.
\end{proof}

\subsection{Main manifold theorem}

Define the finite-horizon rollout parameterization
\begin{equation}
\begin{aligned}
    \Phi_N:\cD_N &\longrightarrow \R^{Nm+(N+1)n},\\
    q &\longmapsto
    \col\bigl(
        u_0,\ldots,u_{N-1},
        x_0(q),x_1(q),\ldots,x_N(q)
    \bigr).
\end{aligned}
\label{eq:Phi_definition}
\end{equation}
By construction,
\begin{equation}
    \im(\Phi_N)=\B_N^{+}.
    \label{eq:image_behavior}
\end{equation}

\begin{theorem}[Finite-horizon nonlinear behavior is an embedded manifold]
\label{thm:main_manifold}
Under Assumption~\ref{ass:smooth_dynamics}, the finite-horizon behavior $\B_N^{+}$ is a $C^r$ embedded submanifold of
\[
    \R^{Nm+(N+1)n}.
\]
Its dimension is
\begin{equation}
    \boxed{\dim \B_N^{+}=n+Nm.}
    \label{eq:behavior_dimension}
\end{equation}
Moreover, $\Phi_N$ is a $C^r$ diffeomorphism from $\cD_N$ onto $\B_N^{+}$ equipped with its embedded-submanifold structure.
\end{theorem}

\begin{proof}
The proof consists of four explicit steps.

\paragraph{Step 1: $\Phi_N$ is $C^r$.}
By Lemma~\ref{lem:rollout_smooth}, every map $q\mapsto x_k(q)$ is $C^r$ on $\cD_N$.  The remaining components of $\Phi_N$ are coordinate projections $q\mapsto u_k$.  Therefore $\Phi_N$ is $C^r$.

\paragraph{Step 2: $\Phi_N$ is injective.}
Suppose $q^a,q^b\in\cD_N$ satisfy
\[
    \Phi_N(q^a)=\Phi_N(q^b).
\]
The vector $\Phi_N(q)$ explicitly contains the initial state $x_0$ and every input $u_0,\ldots,u_{N-1}$.  Equality of the two behavior vectors therefore implies
\[
    x_0^a=x_0^b,
    \qquad
    u_k^a=u_k^b\quad \forall k=0,\ldots,N-1.
\]
By the definition of $q$ in \eqref{eq:q_coordinates}, this gives $q^a=q^b$.  Thus $\Phi_N$ is one-to-one.

\paragraph{Step 3: $\Phi_N$ is an immersion of constant rank $n+Nm$.}
Let
\[
    d:=n+Nm.
\]
There is a fixed permutation matrix $P$ that rearranges the entries of $\Phi_N(q)$ so that the free coordinates $q$ appear first.  Hence
\begin{equation}
    P\Phi_N(q)
    =
    \begin{bmatrix}
        q\\
        G_N(q)
    \end{bmatrix},
    \label{eq:graph_form}
\end{equation}
where
\[
    G_N(q):=\col\bigl(x_1(q),\ldots,x_N(q)\bigr)\in\R^{Nn}.
\]
Differentiating \eqref{eq:graph_form} gives
\begin{equation}
    D(P\Phi_N)(q)
    =
    \begin{bmatrix}
        I_d\\[1mm]
        DG_N(q)
    \end{bmatrix}.
    \label{eq:jacobian_graph}
\end{equation}
The upper block is the $d\times d$ identity.  Therefore the columns of \eqref{eq:jacobian_graph} are linearly independent and
\begin{equation}
    \rank D(P\Phi_N)(q)=d.
\end{equation}
Since a permutation matrix is invertible,
\[
    \rank D\Phi_N(q)=d=n+Nm
\]
for every $q\in\cD_N$.  Thus $\Phi_N$ is an immersion.

\paragraph{Step 4: $\Phi_N$ is a homeomorphism onto its image.}
Define the linear coordinate projection
\begin{equation}
    \pi:\R^{Nm+(N+1)n}\rightarrow\R^{n+Nm}
\end{equation}
that extracts $(x_0,u_0,\ldots,u_{N-1})$ from a behavior vector.  By construction,
\begin{equation}
    \pi\circ\Phi_N=\operatorname{Id}_{\cD_N}
    \label{eq:left_inverse}
\end{equation}
where $\operatorname{Id}_{\cD_N}(q)=q$ denotes the identity map on $\cD_N$. Because $\Phi_N$ is injective, the restriction
\[
    \pi\big|_{\B_N^{+}}:\B_N^{+}\rightarrow\cD_N
\]
is the inverse of $\Phi_N$.  The map $\pi$ is linear and therefore continuous, so its restriction to $\B_N^{+}$ is continuous.  Hence $\Phi_N$ is a homeomorphism between $\cD_N$ and its image $\B_N^{+}$.

We have shown that $\Phi_N$ is a $C^r$ immersion and a homeomorphism onto its image.  It is therefore a $C^r$ embedding.  A standard result in differential topology states that the image of a smooth embedding is an embedded submanifold, with dimension equal to the dimension of the embedding domain \cite[Chs.~4--5]{lee2013manifolds}.  Since $\cD_N$ is open in $\R^d$, its dimension is $d=n+Nm$.  Thus
\[
    \dim\B_N^{+}=n+Nm.
\]
Finally, the inverse map on the image is the coordinate projection $\pi|_{\B_N^{+}}$; in the induced embedded-submanifold structure this inverse is smooth.  Consequently $\Phi_N$ is a $C^r$ diffeomorphism from $\cD_N$ onto $\B_N^{+}$.
\end{proof}

\begin{remark}
The manifold characterization above does not require controllability,
stabilizability, stability, or invertibility of the dynamics.  These
properties become relevant to subsequent questions such as reachability
and closed-loop control, but are not required for the existence of the
finite-horizon behavior manifold itself.
\end{remark}

\begin{remark}[Why controllability is absent]
Controllability enters Willems et al.'s finite-data trajectory-spanning
result---now commonly called the Fundamental Lemma---because that result asks
when the windows of a measured LTI trajectory span the entire
finite-horizon behavior \cite{willems2005pe}.  It is not an assumption needed
for the underlying behavioral definition of a dynamical system.  Theorem~\ref{thm:main_manifold} asks a different question: what is the geometry of the exact behavior generated by a known deterministic recursion?  Since $(x_0,\mathbf u)$ are retained as coordinates of the trajectory itself, no controllability assumption is needed to establish the manifold structure.
\end{remark}

\subsection{Equivalent regular-level-set viewpoint}

The previous proof is constructive and gives an explicit global parameterization.  The same dimension can also be seen directly from the dynamic constraints.

Define
\begin{equation}
    \Gamma_N(\mathbf w)
    :=
    \col\Bigl(
       x_1-f(x_0,u_0),
       \ldots,
       x_N-f(x_{N-1},u_{N-1})
    \Bigr)
    \in\R^{Nn}.
    \label{eq:constraint_map}
\end{equation}
On the open subset of signal space where all evaluations of $f$ are valid,
\begin{equation}
    \B_N^{+}=\Gamma_N^{-1}(0).
\end{equation}
Consider the Jacobian of $\Gamma_N$ with respect to the future-state variables $(x_1,\ldots,x_N)$.  It has the block lower-bidiagonal form
\begin{equation}
\frac{\partial\Gamma_N}{\partial(x_1,\ldots,x_N)}
=
\begin{bmatrix}
I_n & 0 & \cdots & 0\\
-D_xf(x_1,u_1) & I_n & \ddots & \vdots\\
0 & \ddots & \ddots & 0\\
\vdots & \ddots & -D_xf(x_{N-1},u_{N-1}) & I_n
\end{bmatrix}.
\label{eq:block_jacobian}
\end{equation}
This matrix is block triangular with identity blocks on the diagonal, hence it is nonsingular.  Therefore
\[
    \rank D\Gamma_N(\mathbf w)=Nn
\]
at every feasible trajectory.  Thus $0$ is a regular value of $\Gamma_N$.  The regular-level-set theorem then yields an embedded submanifold of codimension $Nn$ \cite[Ch.~5]{lee2013manifolds}, and hence
\begin{align}
    \dim\B_N^{+}
    &=\bigl(Nm+(N+1)n\bigr)-Nn\\
    &=n+Nm.
\end{align}
This gives an independent check of Theorem~\ref{thm:main_manifold}.

\subsection{Tangent space and the linearized finite-horizon behavior}

The tangent space of the behavior manifold has a direct dynamical
interpretation. Let
\[
    \mathbf w=\Phi_N(q)\in\mathfrak B_N^+
\]
be a feasible trajectory, and define along this trajectory
\[
    A_k:=D_x f(x_k,u_k),
    \qquad
    B_k:=D_u f(x_k,u_k),
    \qquad k=0,\ldots,N-1.
\]

\begin{corollary}[Tangent space as the linearized behavior]
\label{cor:tangent_behavior}
Under Assumption~\ref{ass:smooth_dynamics},
\[
T_{\mathbf w}\mathfrak B_N^+
=
\left\{
\operatorname{col}
(\delta u_0,\ldots,\delta u_{N-1},
 \delta x_0,\ldots,\delta x_N)
:
\delta x_{k+1}
=
A_k\delta x_k+B_k\delta u_k
\right\}.
\]
Thus the tangent space at a nonlinear trajectory is precisely the
finite-horizon behavior of the dynamics linearized along that trajectory.
\end{corollary}

\begin{proof}
Since $\Phi_N:\mathcal D_N\rightarrow\mathfrak B_N^+$ is a smooth
embedding, its differential at $q$,
\[
    d\Phi_N|_q:
    T_q\mathcal D_N
    \longrightarrow
    T_{\mathbf w}\mathfrak B_N^+,
\]
maps infinitesimal changes in the free coordinates to infinitesimal changes
of the corresponding trajectory. Because $\mathcal D_N$ is open in
$\R^{n+Nm}$,
\[
    T_q\mathcal D_N\simeq\R^{n+Nm}.
\]
In Euclidean coordinates, the differential $d\Phi_N|_q$ is represented by
the Jacobian matrix $D\Phi_N(q)$. Since $\Phi_N$ is an embedding,
\[
    T_{\mathbf w}\mathfrak B_N^+
    =
    \operatorname{im}D\Phi_N(q).
\]

Let
\[
    \delta q
    =
    \col(\delta x_0,\delta u_0,\ldots,\delta u_{N-1})
    \in T_q\mathcal D_N.
\]
Differentiating the rollout recursion
\[
    x_{k+1}=f(x_k,u_k)
\]
in the direction $\delta q$ gives
\[
    \delta x_{k+1}
    =
    D_xf(x_k,u_k)\delta x_k
    +
    D_uf(x_k,u_k)\delta u_k
    =
    A_k\delta x_k+B_k\delta u_k.
\]
Thus every tangent vector generated by $D\Phi_N(q)$ satisfies the
linearized dynamics.

Conversely, any choice of
$(\delta x_0,\delta u_0,\ldots,\delta u_{N-1})$
uniquely determines $\delta x_1,\ldots,\delta x_N$ through this recursion.
The resulting trajectory variation is precisely
$D\Phi_N(q)\delta q$. Hence the tangent space is exactly the set of
finite-horizon trajectories of the linearized system.
\end{proof}

Using the regular-level-set representation of Section~3.4 gives the
equivalent characterization
\[
    T_{\mathbf w}\mathfrak B_N^+
    =
    \ker D\Gamma_N(\mathbf w).
\]
With the standard Euclidean inner product on the ambient trajectory space,
the corresponding normal space is
\[
    N_{\mathbf w}\mathfrak B_N^+
    =
    \left(\ker D\Gamma_N(\mathbf w)\right)^\perp
    =
    \operatorname{im}D\Gamma_N(\mathbf w)^\top.
\]
Here the transpose uses the Euclidean inner product to identify the
constraint covectors with normal vectors.

\begin{remark}
This is the finite-dimensional discrete-time counterpart of the
continuous-time trajectory-manifold characterization of Hauser and Meyer
\cite{hauser1998trajectory}, in which the tangent space consists of bounded
trajectories of the variational dynamics along the nominal trajectory.
\end{remark}

\section{Consequences for predictive control coordinates}

\subsection{Behavior conditioned on the measured current state}

In a receding-horizon controller, the current state is measured and fixed.  For $\bar x\in\R^n$, define
\begin{equation}
    \B_N^{+}(\bar x)
    :=
    \left\{
        \mathbf w\in\B_N^{+}:x_0=\bar x
    \right\}.
\end{equation}
Let $\cU_N(\bar x)\subseteq\R^{Nm}$ be the set of future input sequences for which the corresponding $N$-step rollout from $\bar x$ is admissible.

\begin{corollary}[Fixed-initial-state behavior]
\label{cor:fixed_state}
Under Assumption~\ref{ass:smooth_dynamics}, if $\cU_N(\bar x)$ is nonempty then it is open and
\[
    \B_N^{+}(\bar x)
\]
is a $C^r$ embedded submanifold parameterized by the future input sequence $\mathbf u$.  Its dimension is
\begin{equation}
    \boxed{\dim\B_N^{+}(\bar x)=Nm.}
\end{equation}
\end{corollary}

\begin{proof}
Fixing $x_0=\bar x$ removes the $n$ free initial-state coordinates from $q$.  The resulting map
\[
    \mathbf u\mapsto
    \col\bigl(\mathbf u,\bar x,x_1(\mathbf u),\ldots,x_N(\mathbf u)\bigr)
\]
has the same graph form as \eqref{eq:graph_form}, now with an $Nm\times Nm$ identity block in its Jacobian.  The proof of Theorem~\ref{thm:main_manifold} therefore applies verbatim with domain dimension $Nm$.
\end{proof}

In the tangent-space characterization of
Corollary~\ref{cor:tangent_behavior}, fixing the initial state simply imposes $\delta x_0=0$.\\

\begin{remark}[Interpretation for online latent optimization]
The full offline behavior has dimension $n+Nm$, while the feasible behavior consistent with a measured current state is an $Nm$-dimensional slice.  An online constraint enforcing $x_0=x(t)$ can therefore be interpreted geometrically as restricting the optimizer from the full behavior manifold to this fixed-initial-state submanifold.
\end{remark}

\subsection{Canonical exact encoder and decoder}

The proof of Theorem~\ref{thm:main_manifold} provides exact global coordinates for the behavior.

\begin{corollary}[Canonical behavior coordinates]
\label{cor:canonical_coordinates}
Define
\begin{align}
    E_\star:\B_N^{+}&\rightarrow\cD_N,
    &E_\star(\mathbf w)&:=\col(x_0,u_0,\ldots,u_{N-1}),\\
    D_\star:\cD_N&\rightarrow\B_N^{+},
    &D_\star(q)&:=\Phi_N(q).
\end{align}
Then
\begin{equation}
    E_\star\circ D_\star=\id_{\cD_N},
    \qquad
    D_\star\circ E_\star=\id_{\B_N^{+}}.
\end{equation}
Hence $E_\star$ and $D_\star$ form an exact $C^r$ encoder--decoder pair for the deterministic behavior.
\end{corollary}

\begin{proof}
The first identity is exactly \eqref{eq:left_inverse}.  For the second, every $\mathbf w\in\B_N^{+}$ is by definition generated by a unique pair $(x_0,\mathbf u)$ because those coordinates are explicitly present in $\mathbf w$.  Re-rolling the deterministic dynamics from those coordinates therefore reproduces the same future states and hence the same $\mathbf w$.
\end{proof}

The maps in Corollary~\ref{cor:canonical_coordinates} are not proposed as a practical model-free controller because $D_\star$ evaluates the unknown dynamics.  Their role is theoretical: they show that a smooth finite-dimensional decoder exists before any neural-network approximation argument is invoked.

\subsection{A lower bound on exact latent dimension}

\begin{corollary}[Latent dimension required for exact differentiable reconstruction]
\label{cor:latent_lower_bound}
Let $p\in\mathbb Z_{>0}$.  Suppose there are $C^1$ maps
\[
    E:\B_N^{+}\rightarrow\R^p,
    \qquad
    D:\mathcal A\subseteq\R^p\rightarrow\R^{Nm+(N+1)n},
\]
with $E(\B_N^{+})\subseteq\mathcal A$ such that
\begin{equation}
    D\circ E=\id_{\B_N^{+}}.
    \label{eq:exact_reconstruction}
\end{equation}
Then necessarily
\begin{equation}
    \boxed{p\ge n+Nm.}
\end{equation}
\end{corollary}

\begin{proof}
Fix $\mathbf w\in\B_N^{+}$.  Differentiate \eqref{eq:exact_reconstruction} along the tangent space $T_{\mathbf w}\B_N^{+}$.  By the chain rule,
\begin{equation}
    \dd D_{E(\mathbf w)}\circ \dd E_{\mathbf w}
    =
    I_{T_{\mathbf w}\B_N^{+}}.
    \label{eq:tangent_identity}
\end{equation}
The right-hand side has rank
\[
    \dim T_{\mathbf w}\B_N^{+}=n+Nm.
\]
The rank of the composition on the left cannot exceed the dimension $p$ of the intermediate latent space.  Therefore
\[
    n+Nm\le p.
\]
\end{proof}

\begin{remark}
Corollary~\ref{cor:latent_lower_bound} is an exact $C^1$ reconstruction statement.  It does not by itself rule out useful lower-dimensional approximate representations over restricted data distributions.  It does show that a globally exact differentiable autoencoder for the full deterministic behavior cannot have latent dimension below the intrinsic behavior dimension.
\end{remark}

\section{Consistency with the LTI behavior used in DeePC}

Consider the LTI special case
\begin{equation}
    x_{k+1}=Ax_k+Bu_k.
    \label{eq:lti_state}
\end{equation}
Then every $x_k$ is a linear function of $(x_0,u_0,\ldots,u_{k-1})$, so $\Phi_N$ is a linear map.  Since Theorem~\ref{thm:main_manifold} shows that $\Phi_N$ is injective,
\begin{equation}
    \B_N^{+}=\im(\Phi_N)
\end{equation}
is a linear subspace with
\begin{equation}
    \dim\B_N^{+}=n+Nm.
\end{equation}
Thus the nonlinear theorem reduces to the familiar linear-behavior geometry: the linear subspace becomes a smooth nonlinear embedded manifold.

The role of the Fundamental Lemma is complementary. For a controllable LTI
behavior and a persistently exciting measured input, sufficiently long Hankel
data span the entire truncated behavior \cite{willems2005pe}. DeePC uses this
result to represent the relevant finite-horizon behavior directly from data
\cite{coulson2019deepc}. Theorem~3.3 addresses a different question: it
characterizes the geometry of the exact nonlinear finite-horizon behavior,
without making a finite-data spanning claim.

\begin{remark}[Output-only behaviors and the DeePC lag]
The clean injectivity argument in Theorem~\ref{thm:main_manifold} uses the fact that $x_0$ itself is retained in the behavior vector.  If only an output $y=h(x,u)$ is retained, two distinct internal states may generate the same finite output window.  Additional observability assumptions are then needed before $(x_0,\mathbf u)$ can serve as coordinates for an input/output behavior.

For minimal LTI systems, DeePC handles exactly this issue through the system lag $\ell(\B)$: an initial input/output trajectory of length $T_{\mathrm{ini}}\ge\ell(\B)$ fixes the compatible state uniquely \cite{coulson2019deepc}.  A nonlinear output-only extension would require an analogous finite-horizon observability/injectivity condition.  No such condition is needed in the full-state setting considered in this note.
\end{remark}

\section{Sampled continuous-time systems}

Many physical plants are naturally described in continuous time.  Consider
\begin{equation}
    \dot x(t)=F\bigl(x(t),u(t)\bigr),
    \label{eq:continuous_system}
\end{equation}
with $x\in\R^n$ and $u\in\R^m$.  Predictive control uses a finite-dimensional input parameterization, most commonly a zero-order hold.  Let the control be constant on each interval of length $\Delta>0$.

For a fixed constant input $u$, let
\begin{equation}
    S_\Delta(x,u)
\end{equation}
denote the state reached after flowing \eqref{eq:continuous_system} for time $\Delta$ from initial state $x$ while holding $u$ fixed.

\begin{assumption}[Smooth continuous-time vector field]
\label{ass:ct}
The vector field $F$ is $C^r$, $r\ge1$, on an open state-input domain, and the initial-state/input pairs under consideration admit a unique solution over the sampling interval $[0,\Delta]$.
\end{assumption}

\begin{lemma}[Smooth sampled transition map]
\label{lem:sampled_map}
Under Assumption~\ref{ass:ct}, the sampled transition map
\[
    S_\Delta:(x,u)\mapsto x(\Delta;x,u)
\]
is $C^r$ on its domain of definition.
\end{lemma}

\begin{proof}
Treat the held input $u$ as a constant parameter of the ODE.  Standard smooth-dependence results for ordinary differential equations state that solutions inherit $C^r$ dependence on initial conditions and parameters from a $C^r$ vector field, on the domain where the corresponding solution exists \cite[Ch.~5]{hartman2002ode}.  Equivalently, one may augment the state with $\dot u=0$ and apply the standard smooth-flow theorem to the augmented vector field.  Evaluating this smooth solution map at time $\Delta$ yields $S_\Delta\in C^r$.
\end{proof}

\begin{corollary}[Trajectory manifold for sampled continuous-time systems]
\label{cor:ct_manifold}
Under Assumption~\ref{ass:ct}, the zero-order-hold sampled system
\begin{equation}
    x_{k+1}=S_\Delta(x_k,u_k)
\end{equation}
has a finite-horizon behavior that is a $C^r$ embedded submanifold of dimension
\begin{equation}
    \boxed{n+Nm.}
\end{equation}
\end{corollary}

\begin{proof}
Lemma~\ref{lem:sampled_map} provides a $C^r$ discrete-time transition map.  Theorem~\ref{thm:main_manifold} applies directly with $f=S_\Delta$.
\end{proof}

\begin{remark}[Relation to continuous-time trajectory manifolds]
The finite-dimensional statement above should be distinguished from the
continuous-time trajectory-manifold construction of
\cite{hauser1998trajectory}. If the control $u(\cdot)$ is allowed to range
over an infinite-dimensional function space, then the corresponding
trajectory space is naturally an infinite-dimensional Banach manifold.
Here, the zero-order-hold assumption replaces $u(\cdot)$ over the horizon
by the finite vector
\[
    \mathbf u=(u_0,\ldots,u_{N-1})\in\mathbb R^{Nm},
\]
which is precisely what yields the finite intrinsic dimension
$n+Nm$. More general finite-dimensional input parameterizations can be
treated in the same way.
\end{remark}

\begin{remark}[Numerical one-step maps]
If training trajectories are generated by a fixed-step numerical integrator such as RK4, one may apply Theorem~\ref{thm:main_manifold} directly to the numerical one-step map.  For a $C^r$ vector field, the classical RK4 update is assembled from finitely many evaluations of $F$, additions, and scalar multiplications, so the resulting discrete update map is itself $C^r$ wherever those evaluations are defined.  The learned behavior is then the trajectory manifold of the discretized dynamics actually used to generate the data.
\end{remark}

\section{How much smoothness should be assumed?}

The preceding results distinguish the smoothness needed for different goals.

\begin{itemize}[leftmargin=2em]
    \item \textbf{$C^1$ dynamics are sufficient for the manifold theorem.}
    If $f\in C^1$, then $\B_N^{+}$ is a $C^1$ embedded manifold and has well-defined tangent spaces that vary continuously.

    \item \textbf{$C^2$ dynamics are a natural working assumption for trajectory-manifold control.}
    If $f\in C^2$, then $\B_N^{+}$ and the canonical rollout parameterization are $C^2$.  This is the appropriate baseline if second derivatives of a decoder are to be used as a smoothness or local-nonlinearity regularizer.

    For $C^2$ dynamics, genuine second-order extrinsic geometry of the embedded
behavior manifold is also well-defined. If
\[
    \mathbf w=\Phi_N(q)
\]
and tangent vectors $v_1,v_2\in T_{\mathbf w}\mathfrak B_N^+$ are written
as
\[
    v_i=D\Phi_N(q)a_i,
\]
then the second fundamental form of the Euclidean embedding is obtained from
the normal component
\[
    \mathrm{II}_{\mathbf w}(v_1,v_2)
    =
    \Pi_{N_{\mathbf w}\mathfrak B_N^+}
    \left(
        D^2\Phi_N(q)[a_1,a_2]
    \right).
\]
Accordingly, the raw Hessian of a particular decoder parameterization is
not, by itself, a coordinate-invariant notion of manifold curvature:
tangential second-derivative components depend on the chosen coordinates.

A related trajectory-specific notion was developed by Notarstefano and
Hauser \cite{notarstefano2008curvature}. They define orthogonality using a
weighted $L_2$ inner product and define a radius of curvature through the
loss of a second-order sufficiency condition for a nearest-trajectory
optimal-control problem. Their construction is therefore an extrinsic,
metric-dependent notion of trajectory-manifold curvature. In subsequent
learning formulations, unprojected decoder-Hessian penalties are more
accurately interpreted as local-nonlinearity regularizers unless they are
explicitly connected to an invariant geometric construction such as a
second fundamental form or a specified nearest-projection geometry. \footnote{The nearest-point viewpoint is related to the classical notion of
\emph{reach} of an embedded subset \cite{federer1959curvature}, defined as
the largest tubular radius on which the nearest-point projection is unique.
The trajectory-manifold radius in
\cite{notarstefano2008curvature} is a local, direction- and
metric-dependent second-order construction and should not in general be
identified with the global reach.}

    \item \textbf{Higher smoothness is inherited but is not required.}
    If the physical transition law is $C^r$ or $C^\infty$, then the finite-horizon behavior has the same regularity.  There is no need to assume $C^\infty$ merely to establish the manifold structure.
\end{itemize}

Accordingly, for the broader theoretical development it is reasonable to adopt $f\in C^2$ on the operating region while noting that Theorem~\ref{thm:main_manifold} itself only needs $C^1$.

\section{Boundaries, constraints, and noise}

The main theorem concerns the unconstrained deterministic behavior over an open admissible domain.  Three extensions should be distinguished.

\paragraph{Input constraints.}
If one restricts the free input coordinates to a closed box, the interior still has the manifold structure above.  Including the box boundary naturally gives a manifold-with-corners structure in the free-coordinate domain, which is carried to the trajectory set by the embedding.  General nonlinear state constraints can create more complicated boundary geometry and require separate regularity conditions.

\paragraph{Process noise.}
With process noise
\[
    x_{k+1}=f(x_k,u_k,\xi_k),
\]
the pair $(x_0,\mathbf u)$ no longer determines a unique trajectory.  If the noise realization $\bm\xi$ is treated as an additional free coordinate, the same graph argument applies with the intrinsic dimension enlarged by the dimension of the noise coordinates.  If noise is not included as a coordinate, measured trajectories should instead be regarded as lying near a deterministic nominal manifold or as samples from a distribution over trajectories.

\paragraph{Measurement noise.}
Measurement noise does not change the deterministic state behavior itself, but observed samples need not lie exactly on it.  This is analogous to the motivation for slack-variable and regularized formulations of DeePC outside the noiseless LTI setting \cite{coulson2019deepc}.

\section{What has been established}

For deterministic nonlinear dynamics with measured state, the first theoretical step of trajectory-manifold control can be summarized as
\begin{equation}
\boxed{
\begin{gathered}
    f\in C^r,\quad r\ge1\\[1mm]
    \Downarrow\\[-1mm]
    \B_N^{+}\text{ is a }C^r\text{ embedded finite-horizon behavior manifold}\\[1mm]
    \dim\B_N^{+}=n+Nm\\[1mm]
    \B_N^{+}\cong_{C^r}\cD_N\subset\R^{n+Nm}\quad\text{through the coordinates }(x_0,\mathbf u).
\end{gathered}}
\end{equation}

This finite-dimensional result is complementary to earlier
continuous-time Banach-manifold descriptions of nonlinear trajectory spaces
\cite{hauser1998trajectory,notarstefano2008curvature}. The distinguishing
feature here is that finite horizon together with a finite-dimensional
input parameterization exposes the global coordinates
$(x_0,\mathbf u)$ explicitly.

The proof is stronger than an abstract manifold-existence assumption: the manifold possesses an explicit global chart given by the initial state and future input sequence.  Consequently, the next learning-theoretic step can be phrased as approximation of the ordinary finite-dimensional rollout map
\[
    \Phi_N:\cD_N\subset\R^{n+Nm}
    \longrightarrow
    \R^{Nm+(N+1)n}
\]
on a compact operating set, rather than approximation of an arbitrary unknown abstract manifold. Data-driven approximation of this rollout map, together with the subsequent predictive-control formulation and its closed-loop stability analysis, is deferred to Part~II.

\bibliographystyle{IEEEtran}
\bibliography{ref}

@inproceedings{coulson2019deepc,
  author    = {Coulson, Jeremy and Lygeros, John and D{\"o}rfler, Florian},
  title     = {Data-Enabled Predictive Control: In the Shallows of the {DeePC}},
  booktitle = {2019 18th European Control Conference (ECC)},
  pages     = {307--312},
  year      = {2019},
  doi       = {10.23919/ECC.2019.8795639},
  url       = {https://doi.org/10.23919/ECC.2019.8795639}
}

@article{willems1979physical,
  author  = {Willems, Jan C.},
  title   = {System theoretic models for the analysis of physical systems},
  journal = {Ricerche di Automatica},
  volume  = {10},
  pages   = {71--106},
  year    = {1979}
}

@article{willems1986timeseries,
  author  = {Willems, Jan C.},
  title   = {From time series to linear system---Part I: Finite dimensional linear time invariant systems},
  journal = {Automatica},
  volume  = {22},
  number  = {5},
  pages   = {561--580},
  year    = {1986},
  doi     = {10.1016/0005-1098(86)90066-X},
  url     = {https://doi.org/10.1016/0005-1098(86)90066-X}
}

@article{willems1991paradigms,
  author  = {Willems, Jan C.},
  title   = {Paradigms and puzzles in the theory of dynamical systems},
  journal = {IEEE Transactions on Automatic Control},
  volume  = {36},
  number  = {3},
  pages   = {259--294},
  year    = {1991},
  doi     = {10.1109/9.73561},
  url     = {https://doi.org/10.1109/9.73561}
}

@article{willems2005pe,
  author  = {Willems, Jan C. and Rapisarda, Paolo and Markovsky, Ivan and De Moor, Bart L. M.},
  title   = {A note on persistency of excitation},
  journal = {Systems \& Control Letters},
  volume  = {54},
  number  = {4},
  pages   = {325--329},
  year    = {2005},
  doi     = {10.1016/j.sysconle.2004.09.003},
  url     = {https://doi.org/10.1016/j.sysconle.2004.09.003}
}

@book{markovsky2006behavioral,
  author    = {Markovsky, Ivan and Willems, Jan C. and Van Huffel, Sabine and De Moor, Bart},
  title     = {Exact and Approximate Modeling of Linear Systems: A Behavioral Approach},
  publisher = {SIAM},
  year      = {2006},
  doi       = {10.1137/1.9780898718263},
  url       = {https://doi.org/10.1137/1.9780898718263}
}

@book{lee2013manifolds,
  author    = {Lee, John M.},
  title     = {Introduction to Smooth Manifolds},
  edition   = {2},
  series    = {Graduate Texts in Mathematics},
  volume    = {218},
  publisher = {Springer},
  year      = {2013},
  doi       = {10.1007/978-1-4419-9982-5},
  url       = {https://doi.org/10.1007/978-1-4419-9982-5}
}

@book{hartman2002ode,
  author    = {Hartman, Philip},
  title     = {Ordinary Differential Equations},
  edition   = {2},
  series    = {Classics in Applied Mathematics},
  volume    = {38},
  publisher = {SIAM},
  year      = {2002},
  note      = {See Chapter 5, ``Dependence on Initial Conditions and Parameters''},
  doi       = {10.1137/1.9780898719222},
  url       = {https://doi.org/10.1137/1.9780898719222}
}

@inproceedings{hauser1998trajectory,
  author    = {John Hauser and David G. Meyer},
  title     = {The Trajectory Manifold of a Nonlinear Control System},
  booktitle = {Proceedings of the 37th IEEE Conference on Decision and Control},
  volume    = {1},
  pages     = {1034--1039},
  year      = {1998},
  doi       = {10.1109/CDC.1998.760833}
}

@inproceedings{hauser2002projection,
  author    = {John Hauser},
  title     = {A Projection Operator Approach to the Optimization of
               Trajectory Functionals},
  booktitle = {15th IFAC World Congress},
  volume    = {35},
  number    = {1},
  pages     = {377--382},
  year      = {2002},
  doi       = {10.3182/20020721-6-ES-1901.00312}
}

@inproceedings{notarstefano2008curvature,
  author    = {Giuseppe Notarstefano and John Hauser},
  title     = {On the Curvature of the Trajectory Manifold of Nonlinear Systems},
  booktitle = {Proceedings of the 47th IEEE Conference on Decision and Control},
  pages     = {1151--1156},
  year      = {2008},
  doi       = {10.1109/CDC.2008.4739487}
}

@article{federer1959curvature,
  author  = {Herbert Federer},
  title   = {Curvature Measures},
  journal = {Transactions of the American Mathematical Society},
  volume  = {93},
  number  = {3},
  pages   = {418--491},
  year    = {1959},
  doi     = {10.2307/1993504}
}

\end{document}